\documentclass[12pt]{amsart}
\usepackage{amssymb}
\usepackage[mathscr]{eucal}
\usepackage{epsf}
\usepackage{epsfig}
\usepackage{color}
\usepackage[notcite, notref]{showkeys}
\DeclareGraphicsExtensions{.pstex,.eps,.epsi,.ps}
\newtheorem{thm}{Theorem}[section]
\newtheorem{cor}[thm]{Corollary}
\newtheorem{lem}[thm]{Lemma}

\theoremstyle{definition}

\theoremstyle{remark}

\numberwithin{equation}{section}

\newcommand{\vol}{\textbf{vol}}
\newcommand{\ric}{\textbf{Rc}}

\newcommand{\e}{\epsilon}

\newcommand{\mF}{\mathcal F}
\begin{document}
\title[Almost splitting for a warped product]{An almost splitting theorem for a warped product space}
\author{Paul W.Y. Lee}
\email{wylee@math.cuhk.edu.hk}
\address{Room 216, Lady Shaw Building, The Chinese University of Hong Kong, Shatin, Hong Kong}
\date{\today}

\begin{abstract}
We prove an almost splitting theorem in the sense of \cite{ChCo} for the warped product space with warped function $f(r)=\cosh\left(r\sqrt{\frac{\lambda}{n-2}}\right)$. 
\end{abstract}

\maketitle

\section{Introduction}

The classical splitting theorem of Cheeger-Gromoll \cite{ChGr} states that a Riemannian manifold with non-negative Ricci curvature that contains a line is isometric to a product of the real line with a submanifold. Here a line is a geodesic which is the image of the real line such that each segment is length minimizing between its end points. 

The above theorem is an example of rigidity results in Riemannian geometry. The assumptions of rigidity results involve inequalities and equalities of certain geometric quantities of the manifold and they conclude that the manifold is isometric to certain model space. 

In \cite{ChCo}, a theory of almost rigidity is developed. It is shown that if an inequality on the Ricci curvature holds and the volume or the diameter is approximately equal to that of the model space, then the manifold is close to the model space in the Gromov-Hausdorff topology. In particular, an almost splitting theorem is shown. It states that if the manifold has approximately non-negative curvature, there are two points $q_0$ and $q_1$ which are far enough away from a point $p$, and the excess function $e(x)=d(x,q_0)+d(x,q_1)-d(q_0,q_1)$ is small at $p$, then the ball centered at $p$ of large radius is close in the Gromov-Hausdorff sense to the corresponding ball in the model. Here the model is the product of the real line and a metric space. 

In this paper, we prove a version of the almost splitting theorem for the warped product space with warped function $f$ given by 
\[
f(r)=\cosh\left(r\sqrt{\frac{\lambda}{n-2}}\right). 
\]
The corresponding rigidity result is obtained in \cite{LiWa1,LiWa2, Le}. In order to state the result, let us introduce some notations. Let us fix a positive function $V$ and let $c$ be the cost function defined by 
\begin{equation}\label{cost}
c(x,y)=\inf_{\gamma\in\Gamma}\int_0^TV(\gamma(t))\,dt, 
\end{equation}
where infimum is taken over all time $T>0$ and all Lipschitz curves $\gamma:[0,T]\to M$ which begins at $x$ and ends at $y$ such that $|\dot\gamma(t)|\leq 1$ for Lebesgue almost all $t$ in $[0,T]$. Here $|\cdot|$ denotes the norm defined by the given Riemannian metric. 

Let $g$ be a positive eigen-function of the Laplace-Beltrami operator with eigenvalue $\lambda$. Let $q_0$ and $q_1$ be two points on the manifold $M$, let $e$ be the excess function corresponding to the cost $c$ defined by $e(x)=c(x,q_0)+c(x,q_1)-c(q_0,q_1)$ with $V=g^{\frac{n-1}{n-2}}$. 

Let $p$ be another point in $M$. For $i=0, 1$, let $b_i(x)=c(x,q_i)-c(p,q_i)$. The model is defined by the wraped product $\mathbb{R}\times_fb_1^{-1}(0)$, where $b_1(0)$ is equipped with the distance function induced by that of $M$. Recall that given two points $x$ and $y$ on a warped product space $\mathbb{R}\times_fN$, the distance between $(r_0,x_0)$ and $(r_1,x_1)$ depends only $r_0$, $r_1$, and the distance of $x_0$ and $x_1$ in $N$. Therefore, this defines a distance function on $\mathbb{R}\times_fb_1^{-1}(0)$. 

\begin{thm}\label{main}
Let $p$ be a point in $M$ and let $v>0$ be a constant such that $B_R(p)\geq vR^n$. For each $\e>0$, there are constants $\e_0>0$ and $L>0$ such that if the followings hold 
\begin{enumerate}
	\item the distance from $q_i$ to the ball $B_R(p)$ of radius $R$ centered at $p$ is greater than $LR$, where $i=0,1$, 
	\item $|f(G^{-1}(b_i))^{2-n}-g|<\e$ on $\partial B_R(p)$, where $i=0,1$ and $G(r)=\int_0^rf^{1-n}$, 
	\item the maximum of $g$ is achieved at a point in $B_R(p)$,
	\item the Ricci curvature $\ric$ satisfies $\ric\geq -\frac{\lambda(n-1)}{n-2}-\frac{\e}{R^2}$,
	\item the sectional curvature is bounded by $C\lambda$ for some positive constant $C$,
	\item $e(p)<\e R$. 
\end{enumerate}
Then there is $R_0(K,\lambda,n)>0$ such that the ball $B_{R_0}(p)$ is $k(\e)$-close in the Gromov-Hausdorff distance to the corresponding ball in the model of radius $R$ centered at $(p,0)$, where $k(\e)\to 0$ as $\e\to 0$ and it depends on $n$, $C$, $\e$, $\lambda$, $\max_{x\in M}g(x)$, and $v$ but not on $R$. 
\end{thm}

The proof of Theorem \ref{main} relies heavily on the ideas from \cite{ChCo}. In section \ref{ET}, we prove an Abresch-Gromoll type inequality using the cost function (\ref{cost}). This is motivated by the work in \cite{Le}. In section \ref{HA}, similar to the almost splitting theorem in \cite{ChCo}, we develop several estimates on the harmonic approximations of $b_i$. The Hessian estimates, in our case, are more complicated due to the involvement of the eigen-function. Using the estimates established in section \ref{harApprox}, we show that the distance function of the Riemannian manifold is close to that of the model. Finally, we finish the proof in section \ref{PO}.

\section*{Acknowledgements}
I would like to thank those who supported me throughout my career.

\section*{Notations}
Thoughout this paper, there are different constants depending on $K$, $\e_0$, $n$, $\lambda$, and $v$. These dependencies will be suppressed throughout the paper. The symbols $k_i(\e)$ and $c_i(\e)$ denote continuous family of constants such that $k_i(\e),c_i(\e)\to 0$ as $\e\to 0$. 

\smallskip

\section{Eikonal Type Equations and Mechanical Hamiltonian Systems}\label{ET}

In this section, we discuss various facts concerning the cost function (\ref{cost}) which are needed for this paper. First, the cost function is a viscosity solution of the following eikonal type equation 
\begin{equation}\label{eikonal}
|\nabla f|_x-V(x)=0. 
\end{equation}
where $|\cdot|$ and $V$ denote, respectively, a Riemannian metric and a positive potential function of a manifold $M$. 

In the Euclidean case, this is shown in \cite{Lions}. One can also consider the above equation as a usual eikonal equation with Riemannian metric given by $\frac{1}{V}|\cdot|$. Therefore, $c$ is also given by the distance function corresponding to this Riemannian metric. However, for later discussion, it is more convenient and natural to think of the cost $c$ as the above optimal control problem. 

An application of the Pontryagin maximum principle (see, for instance, \cite{Ju}) gives 

\begin{thm}\label{Ham}
Let $(T,\gamma)$ be a minimizer of the above minimization problem. Then there is a path $(\gamma(t),p(t))$ in the cotangent bundle $T^*M$ of the manifold $M$ which is a solution to the Hamiltonian system of the Hamiltonian
\begin{equation}
H(x,p)=|p|_x-V(x).
\end{equation}
In particular, $|\dot\gamma(t)|=1$ and 
\begin{equation}\label{minEqn}
\frac{D^2}{dt^2}\gamma=\nabla \log V(\gamma)-\left<\nabla \log V,\dot\gamma\right>\dot\gamma. 	
\end{equation}
Here $\frac{D}{dt}$ denotes the covariant derivative with respect to the given Riemannian metric $|\cdot|$. 
\end{thm}

For convenience, we consider the following map $\Psi_t$ defined by 
\[
\begin{split}
&\Psi_0(v)=x, \quad \frac{d}{dt}\Psi_t(v)\Big|_{t=0}=v,\\
	&\frac{D^2}{dt^2}\Psi_t(v)=\frac{1}{2}\nabla V^2
	(\Psi_t(v)), 
	\end{split}
\]
where $|v|_x=V(x)$. Then minimizers of (\ref{cost}) are of form $\Psi_{r^{-1}(t)}(v)$, where $r(t)=\int_0^tV(\Psi_s(v))ds$. 

The following facts can be obtained using arguments similar to the Riemannian case. 

\begin{lem}\label{costProp}
Assume that $s\in [0,T]\mapsto \gamma(s):=\Psi_{r^{-1}(s)}(v)$ is a minimizer between its end-points. Then
\begin{enumerate}
	\item $\gamma|_{[0,t]}$ is the unique minimizer connecting its end-points $\gamma(0)$ and $\gamma(t)$ for each $t<T$,
	\item there is a neighborhood $U$ of $x$ such that $c_x$ is smooth on $U-\{x\}$, 
	\item $c_x$ is smooth at $\gamma(s)$ for each $s$ in $[0,T)$,
	\item $d(\Psi_{r^{-1}(s)})$ is invertible for each $s$ in $[0,T)$. 
\end{enumerate}
\end{lem}

Next, we state a Laplacian comparison type theorem for the cost function $c$. 

\begin{lem}\label{griccati}
Assume that the Ricci curvature $\ric$ of the manifold is bounded below by a constant $-\frac{\lambda (n-1)}{n-2}-\e$ and the function $g$ satisfies $\Delta g\leq -\lambda g$ and $g\leq K$. Suppose that $V=g^{\frac{n-1}{n-2}}$. Then the Laplacian of $c_x$ satisfies 
\[
	\Delta c_x(\Psi_t(v))\leq \sqrt{\e(n-1)}K^{\frac{n-1}{n-2}}\coth\left(\sqrt{\frac{\e}{n-1}}K^{\frac{n-3}{n-2}}u^{-1}(t)\right),
\]
where $u^{-1}(t)=\int_0^tg(\Psi_s(v))^{\frac{2}{n-2}}ds$. 
\end{lem}

We also need the following volume growth estimate.

\begin{lem}\label{volEst}
Under the assumptions of Lemma \ref{griccati} and that $\Psi_{r^{-1}(s)}$ is contained in $B_1(p)$ for each $s$ in $[\frac{t}{2},t]$, there is a constant \\ $C(\e,K,\mathcal D,n)>0$ such that 
\[
	\frac{\det(d(\Psi_{r^{-1}(t)})_v)}{\det(d(\Psi_{r^{-1}(s)})_v)}\leq C(\e,K,\mathcal D,n,\lambda)
\]
for each $\frac{t}{2}<s<t$, where $\mathcal D=\sup_{x,y\in B_1(p)}\tau(x,y)$. 
\end{lem}

As a consequence, we obtain a version of the Abresch-Gromoll inequality \cite{AbGr,ChCo} in our setting assuming that the potential $V$ is given by $V=\frac{1}{2}g^{\frac{2n-2}{n-2}}$, where $g$ is a positive function satisfying $\Delta g= -\lambda g$. 

\begin{thm}\label{AG}
Assume that the conditions in Lemma \ref{griccati} hold. Let $q_0$ and $q_1$ be two points in the manifold $M$and let $e$ be the excess function defined by
\[
e(y)=c(y,q_0)+c(y,q_1)-c(q_0,q_1).
\]
Assume that, given $\e>0$, there are constants $L(K,\e)$ and $\e_0(K,n,\e)>0$ such that the followings hold:
\begin{enumerate}
\item $d(q_i,B_1(p))\geq L$, where $i=0,1$, 
\item $\ric\geq -\frac{\lambda(n-1)}{n-2}-\e_0$,
\item $e(p)<\e_0$.
\end{enumerate}
Then $e(x)<\e$ for all $x$ in $B_1(p)$. 
\end{thm}

The proof of Theorem \ref{AG} follows closely that of the corresponding result in \cite{AbGr,ChCo}. We give the proof here for the purpose of introducing notations and results needed for later sections. The rest of this section is devoted to the proofs. 

\begin{proof}[Proof of Lemma \ref{griccati}]
Let $w_1,\ldots,w_{n-1}$ be an orthornomal frame of the space $\{v\in T_xM|\,|v|_x=V(x)\}$ and let $w_0=\partial_t$. Let $B(t)$ be the matrix defined by 
\[
	d\Psi_{(t,v)}(w_i)=\sum_{j=0}^{n-1}B_{ij}(t)v_j(t),
\]
where $v_0(t)=\frac{\dot \Psi_t(v)}{|\dot \Psi_t(v)|}$ and $\{v_1(t),\ldots,v_{n-1}(t)\}$ is an orthonormal frame of $v_0(t)^\perp$ such that $\dot v_i(t)$ is contained in $\mathbb{R}v_0(t)$, $i=1,\ldots,n-1$.

Let $E(t)=(v_0(t),\ldots,v_{n-1}(t))^T$ and let $\dot E(t)=\left(\frac{D}{dt}v_0(t),\ldots,\frac{D}{dt}v_{n-1}(t)\right)^T$. It follows that $\dot E(t)=A(t)E(t)$, where 
\[
	A(t)=\left(\begin{array}{cc}
	0 & A_1(t)\\
	-A_1(t)^T & O
	\end{array}\right)
\]
and $A_1(t)=(\begin{array}{ccc}\left<\nabla V(\Psi_t(v)),v_1(t)\right> & \ldots & \left<\nabla V(\Psi_t(v)),v_{n-1}(t)\right>\end{array})$. Therefore, 
\[
	\frac{D}{dt}d\Psi_{(t,v)}(w_i)=\sum_{j=0}^{n-1}\left(\dot B_{ij}(t)+\sum_{k=0}^{n-1}B_{ik}(t)A_{kj}(t)\right)v_j(t). 
\]

By differentiating the above equation again with respect to $t$, it follows that 
\[
\begin{split}
	&\ddot B(t)+2\dot B(t)A(t)+B(t)\dot A(t)+B(t)A(t)^2=-B(t)R(t)+B(t)W(t)
	\end{split}
\]
where $W_{ij}(t)=\frac{1}{2}\nabla^2 V^2(v_i(t),v_j(t))$. 

Let $s(t)$ be the trace of the matrix $B(t)^{-1}\dot B(t)+A(t)$. A computation as in \cite[Section 3]{Le} shows that 
\[
\begin{split}
&\dot s(t)+\frac{s(t)^2}{n-1}-\frac{2s(t)}{n-2}\frac{\left<\nabla g(\Psi_t(v)),\dot\Psi_t(v)\right>}{g(\Psi_t(v))}-\epsilon g(\Psi_t(v))^{\frac{2n-2}{n-2}} \leq 0.
\end{split}
\]

Let $w(t)=g(\Psi_t(v))^{-\frac{2}{n-2}}s(t)$ and $\dot u(t)=\frac{1}{g(\Psi_{u(t)}(v))^{\frac{2}{n-2}}}$. Another computation shows that 
\[
\begin{split}
&\frac{d}{dt} w(u(t))+\frac{w(u(t))^2}{n-1} -\e K^{\frac{2(n-3)}{n-2}}\leq 0.
\end{split}
\]

It follows that $w(u(t))\leq \sqrt{\e(n-1)}K^{\frac{n-3}{n-2}}\coth\left(\sqrt{\frac{\e}{n-1}}K^{\frac{n-3}{n-2}}\,t\right)$. 
\end{proof}

\begin{proof}[Proof of Lemma \ref{volEst}]
We use the same notations as that of the proof of Lemma \ref{griccati}. The function $b(t)=\det(B(t))$ satisfies 
\[
	\frac{d}{dt} \log b(u(t))=s(u(t))g(\Psi_{u(t)}(v))^{-\frac{2}{n-2}}=w(u(t)). 
\]

Let $\bar b(t)=\sinh^{n-1}\left(\sqrt{\frac{\e}{n-1}}K^{\frac{n-3}{n-2}}t\right)$. Then 
\[
\begin{split}
	\frac{d}{dt}\log \bar b(t)&=\sqrt{\e(n-1)}K^{\frac{n-3}{n-2}}\coth\left(\sqrt{\frac{\e}{n-1}}K^{\frac{n-3}{n-2}}t\right)\\
	&\geq \frac{d}{dt} \log b(u(t)). 
	\end{split}
\]

It follows that 
\[
	\begin{split}
	&\frac{b(r^{-1}(t))}{b(r^{-1}(s))}\leq \frac{\bar b(u^{-1}(r^{-1}(t)))}{\bar b(u^{-1}(r^{-1}(s)))}\leq \frac{\bar b(u^{-1}(r^{-1}(t)))}{\bar b(u^{-1}(r^{-1}(t/2)))}\leq C(\e,K,\mathcal D,n). 
	\end{split}
\]

Therefore, 
\[
	\begin{split}
	&\frac{g(\Psi_{r^{-1}(t)})^{\frac{n-1}{n-2}}}{g(\Psi_{r^{-1}(s)})^{\frac{n-1}{n-2}}}\frac{\det(d(\Psi_{r^{-1}(t)}))}{\det(d(\Psi_{r^{-1}(s)}))}\leq C(\e,K,\mathcal D,n). 
	\end{split}
\]

The result follows from the Harnack inequality for $g$ (see, for instance, \cite{Li}).




\end{proof}

\begin{proof}[Proof of Theorem \ref{AG}]
Let $\gamma$ be a minimizer such that $\gamma(0)=q_i$ and $\gamma(u(t))$ is contained in $B_1(p)$. 
A computation shows that  
\[
	\left|\frac{d}{dt}\gamma(u(t))\right|\leq \frac{\left|\dot\gamma(u(t))\right|}{g(\gamma(u(t)))^{\frac{2}{n-2}}}=g(\gamma(u(t)))^{\frac{n-3}{n-2}}\leq K^{\frac{n-3}{n-2}}.
\]

Assume that $L\geq\frac{K^{\frac{n-3}{n-2}}}{\sqrt\e}$. It follows from the assumptions that 
\[
	\frac{K^{\frac{n-3}{n-2}}}{\sqrt \e}\leq d(q_i,\gamma(u(t)))\leq K^{\frac{n-3}{n-2}}t. 
\]
Therefore, by Lemma \ref{griccati}, $\Delta e(\gamma(u(t)))\leq \sqrt{\e}\,C''(K,n)$. 

Let
\[
s_{k}(t)=\frac{\sinh(k\,t)}{k}
\]
and
\[
\varphi_{n,k}(r,l)=\int_r^l\int_t^l\left(\frac{s_k(\tau)}{s_k(t)}\right)^{n-1}d\tau\,dt.
\]

The function $\varphi_{n,k}$ satisfies
\[
\begin{split}
&\partial_r\varphi_{n,k}(r,l)=-\int_r^l\left(\frac{s_k(\tau)}{s_k(r)}\right)^{n-1}d\tau,\\
&\partial_r^2\varphi_{n,k}(r,l)=1+(n-1)\int_r^l\left(\frac{s_k(\tau)}{s_k(r)}\right)^{n-1}\left(\frac{s_k'(r)}{s_k(r)}\right)d\tau,\\
&\partial_r^2\varphi_{n,k}(r,l)+\frac{(n-1)s_k'(r)}{s_k(r)}\partial_r\varphi_{n,k}(r,l)=1. 
\end{split}
\]

Let $\bar\varphi_{s,n,k,\e}$ be the $C^1$ function which is decreasing linearly on $[0,s]$ and equal to $\sqrt\e \,C''\varphi_{n,k}$  on $[s,l]$. It follows that $r\mapsto -\varphi_{n,k}(r,l)$ is increasing and concave. So $y\mapsto-\varphi(d(y,x),l)$ is locally semi-concave on $M-\{x\}$ (see \cite{CaSi}).

Let $x$ be a point in $B_{1}(p)$ such that $e(x)\geq\e_2$ and let $h_{x,s,l,\e}$ be the locally semi-convex function defined on $B_l(x)-\{x\}$ by
$h_{x,s,l,\e}(y)=\bar\varphi_{s,n,k,\e}(d(y,x),l)$. By choosing $\e$ and $s$ small enough, we can assume that $e>h_{x,s,l,\e}$ on $B_{s}(x)$. 

On the other hand, the above computation together with the Laplacian comparison theorem shows that the followings hold in the distributional sense on $B_l(x)-B_{s}(x)$
\[
\begin{split}
&\nabla h_{x,s,l,\e}(y)=\partial_r\varphi_{n,k}(d(y,x),l)\nabla d_{x},\\
&\Delta h_{x,s,l,\e}(y)=\partial_r^2\varphi_{n,k}(d(y,x),l)+\partial_r\varphi_{n,k}(d(y,x),l)\Delta d_{x}\geq 1, 
\end{split}
\]
if $k=\sqrt{\frac{\lambda}{n-2}+\frac{\e}{n-1}}$. Therefore, $\Delta(e -\sqrt\e C'' h_{x,s,l,\e})< 0$ on $B_l(x)-B_{s}(x)$. If there is a point $y_0$ which satisfies $d(x,y_0)=l$ and $e-h_{x,s,l,\e}$ achieves the infimum at $y_0$ among all points $y$ in $B_l(x)-B_{s}(x)$, then $e(y_0)-h_{x,s,l,\e}(y_0)\leq e(y)-h_{x,s,l,\e}(y)$ and
\[
\begin{split}
&\sqrt\e\,C''\varphi_{n,k}(1,l)\leq \sqrt\e\,C''\varphi_{n,k}(d(x,p),l)\\
&\leq h_{x,s,l,\e}(p)-h_{x,s,l,\e}(y_0)\leq e(p)-e(y_0)\leq e(p)<\e.
\end{split}
\]
This gives a contradiction if $\e$ is sufficiently small and the assertion follows. 
\end{proof}

\smallskip

\section{Harmonic approximations and their estimates}\label{HA}

Recall that $b_i(y)=c(y,q_i)-c(p,q_i)$. In this section, we discuss the key estimates involving $b_i$ and its harmonic approximation $\bar b_i$ defined to be the harmonic function which is equal to $b_i$ on the boundary of the ball of radius 1 centered at $p$.

\begin{thm}\label{harApprox}
Under the assumptions of Theorem \ref{AG}, the followings hold:

\begin{enumerate}
\item $|b_i-\bar b_i|<k_1(\e)$ on $B_1(p)$,
\item $\frac{1}{\vol(B_1(p))}\int_{B_1(p)}|\nabla(b_1-\bar b_1)|^2< k_2(\e)$,
\item $|\bar b_1+\bar b_0|<k_3(\e)$ on $B_{1/2}(p)$,
\end{enumerate}
where $k_i(\e)\to 0$ as $\e\to 0$.
\end{thm}

Let $f(r)=\cosh\left(r\sqrt{\frac{\lambda}{n-2}}\right)$ and $G(r)=\int_0^rf^{1-n}$. If $\gamma$ is a unit speed geodesic starting from $p$, then 
\[
	\frac{d}{dt}b_1(\gamma(t))\leq |\nabla\bar b_1|_{\gamma(t)}\leq K^{\frac{n-1}{n-2}}. 
\]

Since $\bar b_1$ is harmonic and equal to $b_1$ on the boundary of $B_1(p)$, $\bar b_1$ is bounded above by $K^{\frac{n-1}{n-2}}$. If $K$ is chosen so that 
\begin{equation}\label{KBd}
K^{\frac{n-1}{n-2}}<\int_0^\infty f^{1-n}, 
\end{equation}
then $f(G^{-1}(\bar b_i))$ is well-defined. Finally, we also assume that there is a constant $v>0$ such that 
\begin{equation}\label{vo}
\vol(B_1(p))\geq v
\end{equation} 
and the constants below could depend on $v$.  

\begin{thm}\label{uEst}
Let $u=f(G^{-1}(\bar b_i))^{2-n}-g$. Under the assumptions of Theorem \ref{AG}, (\ref{KBd}), and (\ref{vo}), 
\[
	\sup_{B_1(p)} |u|\leq k(\e)+\sup_{\partial B_1(p)} |u| 
\]
for some positive constant $k(\e)$ which goes to 0 as $\e\to 0$. 
\end{thm}

As a consequence, 

\begin{cor}\label{uIntEst}
Under the assumptions of Theorem \ref{uEst} and the condition that 
$\sup_{\partial B_1(p)}|u|<\e$,	 
\[
\begin{split}
&\int_{B_1(p)}|\nabla u|^2\leq k(\e)
\end{split}
\]
for some positive constant $k(\e)$ which goes to 0 as $\e\to 0$. 
\end{cor}

Let $F(r)=\int_0^r f$ and $\mF_i =F(G^{-1}(\bar b_i))$. Finally, we obtain the following Hessian estimates. 

\begin{thm}\label{HessianEstI}
Under the assumptions of Theorem \ref{uEst}, 
\[
	\int_{B_{\frac{1}{2}}(p)}\left|\nabla^2\mF_i-\frac{\Delta \mF_i}{n} I\right|^2\leq k(\e)
\]
for some positive constant $k(\e)$ which goes to 0 as $\e\to 0$. 
\end{thm}

As a consequence, 

\begin{cor}\label{HessianEstII}
Under the assumptions of Theorem \ref{uEst}, 
\[
	\int_{B_{\frac{1}{2}}(p)}\left|\nabla^2\mF_i -f'(G^{-1}(\bar b_i))I\right|^2\leq k(\e)
\]
for some positive constant $k(\e)$ which goes to 0 as $\e\to 0$. 
\end{cor}

The rest of this section is devoted to the proof of the above theorems. 

\begin{proof}[Proof of Theorem \ref{harApprox}]
We use the same notations as that of the proof of Theorem \ref{AG}. It follows as in the proof of Theorem \ref{AG} that
\[
\begin{split}
&\Delta \left(b_1-\bar b_1 -\frac{1}{2}\sqrt\e \,C''h_{z,1}\right)< 0,\\
&\Delta \left(b_0+\bar b_1 +e(p) -\frac{1}{2}\sqrt\e \,C''h_{z,1}\right)< 0,
\end{split}
\]
where $h_{z,1}(y)=\varphi_{n,k}(d(y,z),1)$ and $z$ is a point outside $B_1(p)$.

Since $b_1-\bar b_1 -\frac{1}{2}\sqrt\e \,C''h_{z,1}>-\frac{1}{2}\sqrt\e \,C''h_{z,1}\to 0$ on the boundary of $B_1(p)$ as $\e\to 0$, it follows that
\[
b_1-\bar b_1 \geq -c_1(\e)
\]
for some $c_1(\e)>0$. 

On the other hand, 
\[
b_0+\bar b_1 +e(p)-\frac{1}{2}\sqrt\e \,C''h_{z,1}=-b_1+\bar b_1 +e-\frac{1}{2}\sqrt\e \,C''h_{z,1}
\]
is sub-harmonic. Moreover, on the boundary of $B_1(p)$,
\[
-b_1+\bar b_1 +e-\frac{1}{2}\sqrt\e \,C''h_{z,1}\geq -c_2(\e)
\]
for some $c_2(\e)>0$. 

So, by the maximum principle, 
\[
	b_1-\bar b_1 \leq c_2(\e)+e-\frac{1}{2}\sqrt\e \,C''h_{z,1}<c_3(\e). 
\]
The first assertion follows. 

The function $b_1$ is locally semi-concave. Let $\Delta_{\mathcal D}b_1$ denotes its distributional Laplacian which is a measure (see \cite{EvGa}). Let $\Delta b_1$ be the absolutely continuous part of $\Delta_{\mathcal D}b_1$. Since the singular part of $\Delta_{\mathcal D}b_1$ is non-positive,
\[
-\int_{B_1(p)}\Delta b_1\leq -\int_{B_1(p)}\Delta_{\mathcal D} b_1\leq \vol(\partial B_1(p))\leq C_1(\e)\vol(B_1(p))
\]
for some constant $C_1>0$.

On the other hand, if $\Delta b_1=(\Delta b_1)_+-(\Delta b_1)_-$, where $(\Delta b_1)_+$ and $(\Delta b_1)_-$ are the positive and negative parts of $\Delta b_1$, respectively, then
\[
\begin{split}
&\int_{B_1(p)}(\Delta b_1)_-\leq C_1(\e)\vol(B_1(p))+\int_{B_1(p)}(\Delta b_1)_+\\
&\leq C_1(\e)\vol(B_1(p))+\e_3\int_{B_1(p)\cap \Delta b_1>0}1\leq C_2(\e)\vol(B_1(p)).
\end{split}
\]
It follows that $\frac{1}{\vol(B_1(p))}\int_{B_1(p)}|\Delta b_1|\leq C_2(\e)$.

Since $|\nabla b_1|=g^{\frac{n-1}{n-2}}$ a.e. and $\bar b_1-b_1$ vanishes on the boundary of $B_1(p)$,
\[
\begin{split}
&\int_{B_1(p)}|\nabla(\bar b_1-b_1)|^2=\int_{B_1(p)}\Delta_{\mathcal D} b_1(\bar b_1-b_1)\\
&=\int_{B_1(p)}\Delta_{\mathcal D} b_1(\bar b_1-b_1+\e_2+\e_3)-(\e_2+\e_3)\int_{B_1(p)}\Delta_{\mathcal D}b_1\\
&\leq (\e_2+2\e_3)\int_{B_1(p)}|\Delta b_1|+(\e_2+\e_3)C_1(\e)\vol(B_1(p))\\
&\leq (\e_2+2\e_3)C_3(\e)\vol(B_1(p)).
\end{split}
\]


By the first assertion and $|b_0+b_1-\bar b_0-\bar b_1|<k_1(\e)$. By the gradient estimate for harmonic functions \cite{Li} and Theorem \ref{AG}, 
$|\nabla \bar b_0+\nabla \bar b_1|\leq C(\e)|\bar b_0+\bar b_1|<k_2(\e)$ on $B_{1/2}(p)$. 



\end{proof}

\begin{proof}[Proof of Theorem \ref{uEst}]
A computation shows that 
\[
\begin{split}
	&\Delta f(G^{-1}(\bar b_i))^{2-n}=-\lambda f(G^{-1}(\bar b_i))^n|\nabla \bar b_i|^2
	\end{split}
\]
and so 
\[
\begin{split}
	&\Delta (f(G^{-1}(\bar b_i))^{2-n}-g)\\
	&=-\lambda f(G^{-1}(\bar b_i))^n(|\nabla\bar b_i|^2-g^{\frac{2n-2}{n-2}})\\
	&+\lambda f(G^{-1}(\bar b_i))^ng\frac{\left(f(G^{-1}(\bar b_i))^{2-n}\right)^{\frac{n}{n-2}}-g^{\frac{n}{n-2}}}{f(G^{-1}(\bar b_i))^{2-n}-g}(f(G^{-1}(\bar b_i))^{2-n}-g). 
	\end{split}
\]

By Theorem \ref{harApprox}, 
\[
\begin{split}
	&\int_{B_1(p)}\lambda^p f(G^{-1}(\bar b_i))^{pn}(|\nabla\bar b_i|^2-g^{\frac{2n-2}{n-2}})^p\\
	&\leq \lambda^p\int_{B_1(p)}f(G^{-1}(\bar b_i))^{pn}(|\nabla\bar b_i|-g^{\frac{n-1}{n-2}})^2(|\nabla\bar b_i|+g^{\frac{n-1}{n-2}})^{2p-2}\\
	&\leq k(\e)\vol(B_1(p)). 
	\end{split}
\]

We also have 
\[
	0\leq \lambda f(G^{-1}(\bar b_i))^n g\frac{\left(f(G^{-1}(\bar b_i))^{2-n}\right)^{\frac{n}{n-2}}-g^{\frac{n}{n-2}}}{f(G^{-1}(\bar b_i))^{2-n}-g}\leq C.
\]
Therefore, the function $u=f(G^{-1}(\bar b_i))^{2-n}-g$ satisfies a differential equation of the form 
\[
\begin{split}
	&\Delta u=F_1+F_2 u
	\end{split}
\]
where $F_1$ satisfies $\frac{1}{\vol(B_1(p))}\int_{B_1(p)}F_1<k(\e)$ for each fixed $p\geq 1$ and $F_2$ is non-negative and bounded by a constant depending on $\lambda$, $K$, and $n$. Here $k(\e)\to 0$ as $\e\to 0$. 

An argument using Moser iteration as in \cite[Theorem 8.16]{GiTr} using the Sobolev inequality \cite[Theorem 14.2]{Li} gives the result.

\end{proof}

\begin{proof}[Proof of Theorem \ref{HessianEstI}]
Let $\mF_i=h(\bar b_i)$ and $h=F\circ G^{-1}$. A computation using Bochner formula shows that 

\[
\begin{split}
&\frac{1}{2}\Delta(h'(\bar b_i)^2v)\geq\left|\nabla^2\mF_i-\frac{\Delta \mF_i}{n} I\right|^2+\left(\frac{h''(\bar b_i)^2}{n}+h'(\bar b_i)h'''(\bar b_i)\right)|\nabla\bar b_i|^2v\\
&+h'(\bar b_i)h''(\bar b_i)\left<\nabla\bar b_i,\nabla v\right>-\frac{\lambda(n-1)}{n-2}h'(\bar b_i)^2v-\e h'(\bar b_i)^2|\nabla\bar b_i|^2\\
&-\frac{n(n-1)h'(\bar b_i)^2g^{\frac{2n-2}{n-2}}}{(n-2)^2}\left|\nabla \log g+\frac{(n-2)h''(\bar b_i)}{nh'(\bar b_i)}\nabla\bar b_i\right|^2
\end{split}
\]

Let $h=F\circ G^{-1}$, where $F(r)=\int_0^r f(s)ds$ and $G(r)=\int_0^rf(s)^{1-n}ds$. It follows that 


\[
\begin{split}
&\frac{1}{2}\Delta(h'(\bar b_i)^2v)\geq\left|\nabla^2\mF_i-\frac{\Delta \mF_i}{n} I\right|^2+\left(\frac{h''(\bar b_i)^2}{n}+h'(\bar b_i)h'''(\bar b_i)\right)|\nabla\bar b_i|^2v\\
&+h'(\bar b_i)h''(\bar b_i)\left<\nabla\bar b_i,\nabla v\right>-\frac{\lambda(n-1)}{n-2}h'(\bar b_i)^2v-\e h'(\bar b_i)^2|\nabla\bar b_i|^2\\
&-\frac{n(n-1)h'(\bar b_i)^2g^{\frac{2}{n-2}}}{(n-2)^2}\left|\nabla u +(n-2)f'(G^{-1}(\bar b_i))f(G^{-1}(\bar b_i))^{n-2}u\nabla\bar b_i\right|^2,
\end{split}
\]
where $u=f(G^{-1}(\bar b_i))^{2-n} -g$ and $v=|\nabla\bar b_i|^2-g^{\frac{2n-2}{n-2}}$.

By multiplying the above inequality by a cut-off function, which equals to $1$ on $B_{\frac{1}{2}}(p)$ and supported in $B_1(p)$, integrating over $B_1(p)$, and applying Theorem \ref{uEst}, Corollary \ref{uIntEst}, and Theorem \ref{harApprox}, the result follows. 
\end{proof}

\begin{proof}[Proof of Corollary \ref{HessianEstII}]
Let $H(r)=r^{\frac{2n-2}{n-2}}$
\[
\begin{split}
	&\Delta \mF_i =h''(\bar b_i)v+nf'(G^{-1}(\bar b_i))\\
	&+nf(G^{-1}(\bar b_i))^{2n-2}f'(G^{-1}(\bar b_i))(H(g)-H(f(G^{-1}(\bar b_i))^{2-n}))\\
	\end{split}
\]

The result follows from this, Theorem \ref{HessianEstI}, Theorem \ref{harApprox}, and Theorem \ref{uEst}. 
\end{proof}

\smallskip

\section{Distance Estimate}\label{DE}

In this section, we prove that the distance function is close to that of the model space using the estimates obtained from the previous sections. More precisely, let $x_m,y_m,z_m$ be three points in the unit ball $B_{1,m}(p)$ of the model space such that $c_m(x,z)=b_{1,m}(z)-b_{1,m}(x)$. Here $c_m$ and $b_{1,m}$ denote the functions $c$ and $b_1$ in the case of the model. For the rest of this work, the subsript $m$ will be reserved for quantities in the model space. 


\begin{thm}\label{disEst}
Let $x,y,z$ be points in the ball $B_1(p)$ such that $b_1(z)-b_1(x)=c(x,z)=c_m(x_m,z_m)$, $b_1(x)=b_{1,m}(x_m)$, $b_1(y)=b_{1,m}(y_m)$, and , $b_1(z)=b_{1,m}(z_m)$. Then, under the assumptions of Theorem \ref{uEst}, that the sectional curvature on $B_1(p)$ is bounded by $C\lambda$ for some positive constant $C$, and that the maximum $K$ of $g$ is achieved in $B_1(p)$, there is a constant $k(\e)>0$ such that $k(\e)\to 0$ as $\e\to 0$ and $|d(y,z)-d_m(y_m,z_m)|<k(\e)$. 
\end{thm}

In the above theorem, $k(\e)$ depends on the bounds of the sectional curvature which is suppressed. 


For the proof of Theorem \ref{disEst}, we need the following lemmas. 

\begin{lem}\label{integralAngle}
Suppose that the assumptions of Theorem \ref{disEst} hold. Let 
\[
\begin{split}
	&\mathcal Q_i(y,z)=\Big|\left<\nabla d_y,\nabla \mF_i\right>_{z}+\frac{\mF_i(z)-\mF_i(y)}{d(y,z)}\\
	&+\frac{1}{d(y,z)}\int_0^{d(y,z)}\int_s^{d(y,z)}f'(G^{-1}(\bar b_i(\gamma_{y,z}(s')))\,ds'\,ds\Big|, 
\end{split}
\]
Then there is a constant $k(\e)>0$ such that 
\[
\vol(V_{1})>(1-k(\e)^{1/2})\vol(B_{1/2}(p)),
\]
where
\[
\begin{split}
&D_{1}(y,y_0)=\left\{y_1\in B_{r}^c(p)\Big|\int_0^{\tau(y_0,y_1)}\mathcal Q_i(y,\gamma_{y_0,y_1}(t))\,dt <k(\e)^{1/6}\right\},\\
&Q_{1}(y)=\{y_0\in B_{r}^c(p)|\vol(D_{1}(y,y_0))\geq (1-k(\e)^{1/6})\vol(B^c_{r}(p))\},\\
&V_{1}=\{y\in B_{\frac{1}{4}}(p)|\vol(Q_{1}(y))\geq (1-k(\e)^{1/6})\vol(B_r^c(p))\},\\
&k(\e)\to 0 \text{ as } \e\to 0.
\end{split}
\]
Here $B_r^c(p)$ denotes the ball of radius $r=\frac{CK^{\frac{n-1}{n-2}}}{8}$ centered at $p$ defined using the cost function $c$. 
\end{lem}

Let $U_0$ and $U_1$ be two subsets of $M$. Let $U$ be the set of points of the form $\gamma(t)$, where $\gamma$ is a minimizer of (\ref{cost}) which starts from a point in $U_0$ and ends at $U_1$. Let $\gamma_{y_0,y_1}(t)$ be the minimizer of (\ref{cost}) connecting $y_0$ and $y_1$ which is defined Lebesgue almost everywhere on $M\times M$. The proof of the following lemma can be proved using Lemma \ref{volEst} and the arguments in \cite[Theorem 2.11]{ChCo}. 

\begin{lem}\label{integral}
Assume that $P$ is a non-negative measurable function on $M$. Then
\[
\begin{split}
&\int_{U_0\times U_1}\int_0^{\tau(x,y)}P(\gamma_{x,y}(s))\,ds\,d(x,y)\\
&\leq C(\e,\mathcal D)(T(U_0,U_1)\vol(U_0)+T(U_1,U_0)\vol(U_1))\int_UP,
\end{split}
\]
where $\gamma_{y_0,y_1}$ is the minimizer connecting $y_0$ and $y_1$, $\tau(y_0,y_1)$ is the length of $\gamma_{y_0,y_1}$, $T(\dot\gamma_{y_0,y_1}(0))$ is the set of time $t$ such that $\Psi_{r^{-1}(t)}(v)\in U_1$ and $s\in[0,t]\mapsto\Psi_{r^{-1}(s)}(v)$ is a minimizer, $T(U_0,U_1)=\sup_{x\in U_0,|v|_x=1}|T(v)|$, and $\mathcal D=\sup_{x_i\in U_i}\tau(x_0,x_1)$.
\end{lem}

We will need the following lemma which can be proved using an argument in \cite{Co1}. 

\begin{lem}\label{flowintegral}
Let $\Psi_t(x,v)$ be the geodesic flow defined on the unit tangent bundle $SM$. Assume that $P$ is a non-negative measurable function on the unit tangent bundle $SM$ and $\frac{1}{\vol(U)}\int_{B_l(U)}\sup_{|v|_x=1}P(x,v)\,dx<\e$, where $U$ is a subset of $M$ and $B_l(U)$ is a neighborhood of $U$ of radius $l$. Then
\[
\begin{split}
	&\frac{1}{l\,\vol(SU)}\int_{SU}\int_0^{l}P(\Psi_t(x,v))\,dt\,d(x,v)<\e.
\end{split}
\]
\end{lem}

Similar fact holds for the Hamiltonian flow $\Psi_t^c$ of the Hamiltonian $H(x,v)=|v|_x-g(x)^{\frac{n-1}{n-2}}$ restricted to the set $S^cU=\{(x,v)\in TU|H(x,v)=0\}$. 

\begin{lem}\label{flowcintegral}
Assume that $P$ is a non-negative measurable function on the unit tangent bundle $SM$ and $\frac{1}{\vol(U)}\int_{B_l(U)}\sup_{|v|_x=1}P(x,v)\,dx<\e$, where $U$ is a subset of $M$ and $B_l(U)$ is a neighborhood of $U$ of radius $l$. Then, there is a constant $C>0$ such that 
\[
\begin{split}
	&\frac{1}{l\,\vol(S^cU)}\int_{S^cU}\int_0^{l}P(\Psi_t^c(x,v))\,dt\,d(x,v)<C\e.
\end{split}
\]
\end{lem}

The rest of this section is devoted to the proofs.

\begin{proof}[Proof of Lemma \ref{integralAngle}]
Since $c(p,y)\leq K^{\frac{n-1}{n-2}}d(p,y)$, $B_r(p)$ is contained in $B^c_{K^{\frac{n-1}{n-2}}r}(p)$. On the other hand, by the Harnack inequality, $g^{\frac{n-1}{n-2}}\geq C(\lambda,n,\e)K^{\frac{n-1}{n-2}}$ on $B_{\frac{1}{2}}(p)$. It follows that $B_r^c(p)$ is contained in $B_{\frac{r}{CK^{\frac{n-1}{n-2}}}}(p)$. 

Let $\gamma_{y_0,y_1}$ and $\bar\gamma_{y_0,y_1}$ denotes the minimizer and geodesic which connect $y_0$ and $y_1$, respectively, which are well-defined Lebesgue almost everywhere. It follows that 
\[
\begin{split}
	&\mathcal Q(y,z)=\frac{1}{t}\int_0^t-\left<\nabla\mF_i(\bar\gamma_{y,z}(t)),\dot{\bar\gamma}_{y,z}(t)\right>\\
	&+\left<\nabla\mF_i(\bar\gamma_{y,z}(s)),\dot{\bar\gamma}_{y,z}(s)\right>+\int_s^tf'(G^{-1}(b(\bar\gamma_{y,z}(s')))\,ds'\,ds\\
	&=-\frac{1}{t}\int_0^t\int_s^t\left<\nabla^2\mF_i(\bar\gamma_{y,z}(s'))\dot{\bar\gamma}_{y,z}(s'),\dot{\bar\gamma}_{y,z}(s')\right>-f'(G^{-1}(b(\bar\gamma_{y,z}(s'))))\,ds'\,ds\\
	&\leq\int_0^t|\nabla^2\mF_i-f'(G^{-1}(\bar b_i))I|_{\bar\gamma_{y,z}(s)}\,ds. 
	\end{split}
\]

Therefore, by Lemma \ref{integral} and Corollary \ref{HessianEstII}, 

\[
\begin{split}
	&\frac{1}{\vol(B_r^c(p))^2\vol(B_{\frac{1}{4}}(p))}\int_{B_r^c(p)^2}\int_{B_{\frac{1}{4}(p)}}\int_0^{\tau(y_0,y_1)}Q(y,\gamma_{y_0,y_1}(t))\,dt\,dy\,dy_0\,dy_1\\
	&\leq\frac{C_1}{\vol(B_{\frac{1}{4}}(p))^2}\int_{B_{\frac{1}{4}}(p)^2}Q(y,z)\,dy\,dz\\
	&\leq\frac{C_2}{\vol(B_{\frac{1}{2}}(p))}\int_{B_{\frac{1}{2}}(p)}|\nabla^2\mF_i - f'(G^{-1}(\bar b_i))|_{z}\,dz<k(\e).
\end{split}
\]
The assertion follows. 


\end{proof}

\begin{proof}[Proof of Theorem \ref{disEst}]
Recall that $c_i(\e)$ and $k_i(\e)$ are positive continuous functions of $\e$ which $\to 0$ as $\e\to 0$. 
The ball $B_{R_0}(p)$ with $R_0=\frac{r}{K^{\frac{n-1}{n-2}}}$ is contained in the ball $B_r^c(p)$. By Theorem \ref{integralAngle}, Lemma \ref{integral}, Theorem \ref{harApprox}, the volume comparison theorem, and the assumption (\ref{vo}), there are positive $c_i(\e)\to 0$ as $\e\to 0$ ($i=1,2,3$), such that the followings hold: given any three points $x,y,z$ in $B_{R_0}(p)$ there are points $x',y',z'$ in $B_{c_1(\e)}(x),B_{c_2(\e)}(y),B_{c_3(\e)}(z)$, respectively, such that 
\begin{equation}\label{small1}
	\int_0^{\tau(x',z')}|\nabla b_i-\nabla\bar b_i|^2_{\gamma_{x',z'}(t)}dt<k_1(\e),
\end{equation}
\begin{equation}\label{small2}
	\int_0^{d(x',z')}\mathcal Q(y',\bar\gamma_{x',z'}(t))dt<k_2(\e).
\end{equation}
\begin{equation}\label{small3}
	\int_0^{d(x',z')}|\nabla^2\mF_i-f'(G^{-1}(\bar b_i))|_{\bar\gamma_{x',z'}(t)}dt<k_3(\e),
\end{equation}
\begin{equation}\label{small4}
	\int_0^{\tau(x',z')}|\nabla^2\mF_i-f'(G^{-1}(\bar b_i))|_{\gamma_{x',z'}(t)}dt<k_4(\e).
\end{equation}



Let us fix a time $s$ in $[0,\tau(x',z')]$ and let $(y',v')$ be the tangent vector such that $\Psi_l(y',v')=\bar\gamma_{x',z'}(s)$. By Corollary \ref{HessianEstII} and Lemma \ref{flowintegral}, there is a point $(y'',v'')$ in $B_{c_4(\e)}(y',v')$ such that
\[
\begin{split}
&\int_0^{1/4}\left|U''(t)-\frac{\lambda}{n-2}U(t)\right|dt\\
&\leq\int_0^{1/4}|\nabla^2\mF_i - f'(G^{-1}(\bar b_i))I|_{\bar\gamma_{y'',p}(t)}dt<k_4(\e)
\end{split}
\]
where $U(t)=F(G^{-1}(\bar b_i(\Psi_t(y'',v'')))$. Here the ball $B_{c_4(\e)}(y',v')$ is defined by the distance on the unit tangent bundle $SM$ induced by the Riemannian metric on $M$ and its Levi-Civita connection. 

Let $\bar U_{a_1,a_2,T}$ be the solution of the equation 
\begin{equation}\label{U}
\bar U''(t)=\frac{\lambda}{n-2}\bar U(t)
\end{equation} with boundary conditions $\bar U(0)=a_1$ and $\bar U(T)=a_2$. 



Since the sectional curvature is bounded and the two points $(y',v')$ and $(y'',v'')$ are $c_4(\e)$-close, $\Psi_t(y',v')$ and $\Psi_t(y'',v'')$ are $k_{8}(\e)$-close. It follows from an argument using Gronwall's inequality that $|U(t)- \bar U_{U(0),U(l),l}(t)|<k_5(\e)$ and $|U'(t)- \bar U'_{U(0),U(l),l}(t)|<k_6(\e)$. It also follows that 
\begin{equation}\label{re1}
|\bar U_{U(0),U(l),l}(t)-F(G^{-1}(\bar b_i(\bar\gamma_{y,\bar\gamma_{x',z'}(s)}(t))))|<k_{7}(\e). 
\end{equation}

By the same argument applied to (\ref{small3}), we also have 
\begin{equation}\label{re2}
|\bar U_{\mF_i(x'),\mF_i(z'),\tau(x',z')}(t)-\mF_i(\gamma_{x',z'}(t))|<k_{8}(\e)
\end{equation}
and 
\begin{equation}\label{re3}
|\bar U'_{\mF_i(x'),\mF_i(z'),\tau(x',z')}(t)-\left<\nabla\mF_i(\gamma_{x',z'}(t)),\dot{\gamma}_{x',z'}(t))\right>|<k_{9}(\e). 
\end{equation}

By Theorem \ref{uEst}, 

\[
\begin{split}
	&\frac{d}{dt}\left(G^{-1}(b_i(\gamma_{x,z}(t)))-t\right)\\
	&=\frac{|\nabla b_i|_{\gamma_{x,z}(t)}}{G'(G^{-1}(b_i(\gamma_{x,z}(t))))}-1\\
	&<k_3(\e). 
\end{split}
\]
It follows from this and Theorem \ref{harApprox} that 
\begin{equation}\label{GLin1}
\begin{split}
	&|G^{-1}(\bar b_i(\gamma_{x,z}(t)))-G^{-1}(\bar b_i(x))-t|<k_{10}(\e)t
\end{split}
\end{equation}
and so
\begin{equation}\label{GLin2}
\begin{split}
	&|\mF_i(\gamma_{x,z}(t))-F(G^{-1}(\bar b_i(x))+t)|<k_{11}(\e)t. 
\end{split}
\end{equation}

Since $F(G^{-1}(\bar b_i(x))+t)$ is a solution of (\ref{U}) and, by (\ref{GLin2}), the boundary values are close to $\mF_i(x')$ and $\mF_i(z')$, it follows from (\ref{re2}) and (\ref{re3}) that 
\begin{equation}\label{re4}
|F(G^{-1}(\bar b_i(x'))+t)-\mF_i(\gamma_{x',z'}(t))|<k_{12}(\e),
\end{equation}
and 
\begin{equation}\label{re5}
|f(G^{-1}(\bar b_i(x'))+t)-\left<\nabla\mF_i(\gamma_{x',z'}(t)),\dot{\gamma}_{x',z'}(t)\right>|<k_{13}(\e). 
\end{equation}

By combining (\ref{re4}) and (\ref{re5}), 
\[
|f(G^{-1}(\bar b_i(\gamma_{x',z'}(t))))-\left<\nabla\mF_i(\gamma_{x',z'}(t)),\dot{\gamma}_{x',z'}(t)\right>|<k_{14}(\e). 
\]

It follows from this and (\ref{uEst}) that 

\[
	||\nabla b_i|_{\gamma_{x',z'}(t)}-\left<\nabla\bar b_i(\gamma_{x',z'}(t)),\dot{\gamma}_{x',z'}(t)\right>|<k_{15}(\e). 
\]

Therefore, by (\ref{small1}), 
\[
	\int_0^{\tau(x',z')}||\nabla \bar b_i|_{\gamma_{x',z'}(t)}-\left<\nabla\bar b_i(\gamma_{x',z'}(t)),\dot{\gamma}_{x',z'}(t)\right>|dt<k_{16}(\e). 
\]

It follows that 
\begin{equation}\label{re6}
	\int_0^{\tau(x',z')}\left|\dot\gamma_{x',z'}(t)-\frac{\nabla\bar b_i(\gamma_{x',z'}(t))}{|\nabla \bar b_i|_{\gamma_{x',z'}(t)}}\right|dt<k_{17}(\e). 
\end{equation}




Let 
\[
\begin{split}
	\mathcal G(t,r_1,r_2,l)&=\frac{F(r_1+t)-F(r_2)}{l\,f^n(r_1+t)}\\
	&+\frac{1}{l\,f^n(r_1+t)}\int_0^l\int_s^lf'(F^{-1}(\bar U(s'))\,ds'\,ds.
\end{split}
\]
Here the dependencies of $\bar U$ on $r_1$ and $\tau(x,z)$ are suppressed. 


Assume that the distance from $\gamma$ to $y$ is greater than $c_5(\e)$. It follows from (\ref{small2}), (\ref{re1}), (\ref{re4}), and (\ref{re6}) that if $c_5$ is appropriately chosen, then 
\[
\begin{split}
	&\int_0^{\tau(x',z')}\Big|L'(t)+\mathcal G(t,G^{-1}(\bar b_i(x')),G^{-1}(\bar b_i(y')),L(t))\Big|dt<k_{18}(\e).
\end{split}
\]
Let $L_m$ be the corresponding quantity in the model. It follows that $L_m'(t)+\mathcal G(t,G^{-1}(\bar b_i(x')),G^{-1}(\bar b_i(y')),L_m(t))=0$ and 


\[
\begin{split}
	&|L(t)-L_m(t)|\leq\int_0^{t}|L'(s)-L_m'(s)|ds\\
	&\leq\int_0^{t}|L'(s)-\mathcal G(s,L(s))+\mathcal G(s,L(s))-\mathcal G(s,L_m(s))|ds\\
	&\leq k_{11}(\e)+C\int_0^{t}|L(s)-L_m(s)|ds 
\end{split}
\]
It follows from Gronwall's inequality that $|L(t)-L_m(t)|<k_{12}(\e)$. 

Suppose the distance from $\gamma$ to $y$ is less than $c_5(\e)$. In this case, it is enough to show that $|\tau(x,z)-d(x,z)|<k_{13}(\e)$. Since $\tau(x,z)$ is the length of the minimizer $\gamma_{x,z}$, $\tau(x,z)\geq d(x,z)$. 




It also follows from Theorem \ref{uEst}, (\ref{GLin2}), and (\ref{re4}) that 
\[
\begin{split}
&\int_0^{\tau(x',z')}f(G^{-1}(\bar b_i(\gamma_{x',z'}(t))))^{2-n}dt \\
	&=c(x',z')\leq\int_0^{d(x',z')}g(\bar\gamma_{x',z'}(t))dt\\
	&\leq k(\e)+\int_0^{d(x',z')}f(G^{-1}(\bar b_i(\bar\gamma_{x',z'}(t))))^{2-n}dt\\
	&\leq \tilde k(\e)+\int_0^{d(x',z')}f(G^{-1}(\bar b_i(\gamma_{x',z'}(t))))^{2-n}dt 
\end{split}
\]

It follows that $|d(x,z)-\tau(x,z)|<k(\e)$ as claim.

\end{proof}

\smallskip

\section{Proof of Theorem \ref{main}}\label{PO}

In this section, we finish the proof of Theorem \ref{main}. By scaling, it is enough to consider the case when $R=1$. Let $R_0>0$ be a small enough constant such that all the estimates in the previous sections hold on $B_{R_0}(p)$. 

By the third assertion of Theorem \ref{harApprox}, we can find a subset $W=\{x_1,\ldots,x_N\}$ in $B_{R_0}(p)$ which is $c_1(\e)$-dense. Suppose that $b_1(x_i)>0$. Let $y_i$ be the point in $b_1^{-1}(0)$ such that $b_1(x_i)-b_1(y_i)=c(x_i,y_i)$. Let $x_{i,m}$ be the point in the model $\mathbb{R}\times_fb_1^{-1}(0)$ such that $b_m(x_i)-b_m(y_i)=c(x_i,y_i)$ ($b_1=-b_0$ in the case of the model and it is denoted by $b_m$). If $b_1(x_1)<0$ instead, then one can move $x_i$ along the flow of $-\nabla b_0$ (recall that $b_0$ is differentiable at $x_i$). It follows that there is a point $y_i$ in $b_1^{-1}(0)$ such that $b_0(x_i)-b_0(y_i)=c(x_i,y_i)$. 

It remains to show that $\{x_{1,m},\ldots,x_{N,m}\}$ is $c_2(\e)$-dense in $B_{R_0,m}(p)$. It follows from this, Theorem \ref{disEst}, and \cite[10.1.1]{Pe} that $B_{R_0}(p)$ and $B_{R_0,m}(p)$ are $c_3(\e)$-close in the Gromov-Hausdorff distance. 

Let $x_m=(b_m(x_m), y_m)$ be a point in $B_{R_0,m}(p)$ of the warped product model. Assume that $b_m(x_m)>0$. By assumption $b_1(y_m)=0$. Let $y$ be a point in $W$ which is $c_1(\e)$-close to $y_m$ in $M$. Let $x$ be the point in $M$ which satisfies $b_0(y)-b_0(x)=b_m(x_m)=c(x,y)$. Let $x_i$ be a point in $W$ which is $c_1(\e)$-close to $x$. By applying Theorem \ref{disEst} twice, it follows that $x_{i,m}$ is $c_2(\e)$-close to $x_m$ as claimed. Similar procedure works if $b_m(x_m)<0$. 



\smallskip


\begin{thebibliography}{100}
\bibitem{AbGr} U. Abresch, D. Gromoll: On complete manifolds with nonnegative Ricci curvature. J. Amer. Math. Soc. 3 (1990), no. 2, 355-374. 
\bibitem{CaSi} P. Cannarsa, C. Sinestrari: Semiconcave functions, Hamilton-Jacobi equations, and optimal control. Progress in Nonlinear Differential Equations and their Applications, 58. Birkh\"auser Boston, Inc., Boston, MA, 2004.
\bibitem{ChCo} J. Cheeger, T.H. Colding: Lower bounds on Ricci curvature and the almost rigidity of warped products. Ann. of Math. (2) 144 (1996), no. 1, 189-237.
\bibitem{ChGr} J. Cheeger, D. Gromoll: The splitting theorem for manifolds of nonnegative Ricci curvature. J. Differential Geometry 6 (1971/72), 119-128. 
\bibitem{Co1}  T.H. Colding: Shape of manifolds with positive Ricci curvature. Invent. Math. 124 (1996), no. 1-3, 175-191.
\bibitem{EvGa} L.C. Evans, R.F. Gariepy: Measure theory and fine properties of functions. Revised edition. Textbooks in Mathematics. CRC Press, Boca Raton, FL, 2015.
\bibitem{GiTr}  D. Gilbarg, N.S. Trudinger: Elliptic partial differential equations of second order. Reprint of the 1998 edition. Classics in Mathematics. Springer-Verlag, Berlin, 2001.
\bibitem{Ju} V. Jurdjevic: Geometric control theory. Cambridge Studies in Advanced Mathematics, 52. Cambridge University Press, Cambridge, 1997.
\bibitem{Le} P.W.Y. Lee: A warped product splitting theorem through weak KAM theory. arxiv:1712.08896
\bibitem{Li} P. Li: Geometric analysis. Cambridge Studies in Advanced Mathematics, 134. Cambridge University Press, Cambridge, 2012.
\bibitem{LiWa1} P. Li, J. Wang: Complete manifolds with positive spectrum. J. Differential Geom. 58 (2001), no. 3, 501-534.
\bibitem{LiWa2} P. Li, J. Wang: Complete manifolds with positive spectrum. II. J. Differential Geom. 62 (2002), no. 1, 143-162.
\bibitem{Lions} P.L. Lions: Generalized solutions of Hamilton-Jacobi equations. Research Notes in Mathematics, 69. Pitman (Advanced Publishing Program), Boston, Mass.-London, 1982. 
\bibitem{On} B. O'Neil: The fundamental equations of a submersion. Michigan Math. J. 13 (1966), 459-469.
\bibitem{Pe} P. Petersen: Riemannian geometry. Third edition. Graduate Texts in Mathematics, 171. Springer, Cham, 2016.
\bibitem{Vi} C. Villani: Optimal transport. Old and new. Grundlehren der Mathematischen Wissenschaften [Fundamental Principles of Mathematical Sciences], 338. Springer-Verlag, Berlin, 2009.
\end{thebibliography}
\end{document}